\documentclass[10pt, twocolumn]{article}
\usepackage[utf8]{inputenc}
\usepackage{authblk}
\usepackage[top=2cm, bottom=4.5cm, left=2cm, right=2cm]{geometry}
\usepackage{graphicx}
\usepackage{amsmath,amsfonts,amssymb}
\usepackage{listings}
\usepackage[T1]{fontenc}
\usepackage{listings}
\usepackage{color}
\usepackage[backref=page]{hyperref}
\usepackage{url}
\usepackage{caption}
\usepackage{subcaption}
\usepackage{dirtytalk}
\usepackage{amsthm}
\usepackage[capitalize]{cleveref} 
\usepackage{cancel}    
\usepackage{url}            
\usepackage{booktabs}       
\usepackage{microtype}      
\usepackage{setspace}
\usepackage{amssymb}
\usepackage{url}
\usepackage{lettrine}
\usepackage{framed}
\usepackage{color}
\usepackage{algorithm}
\usepackage{algorithmic}
\usepackage{overpic}
\usepackage{dsfont}
\usepackage{booktabs} 
\usepackage{framed}
\usepackage{graphicx}
\usepackage{dsfont}
\usepackage{framed}
\usepackage{bm}
\usepackage{bbm}
\usepackage{overpic}
\usepackage[svgnames]{xcolor}
\newtheorem{theorem}{Theorem}
\newtheorem{proposition}{Proposition}
\newtheorem{lemma}{Lemma}
\usepackage{xr}

\renewcommand{\vec}[1]{\bm{#1}}

\newcommand{\etal}{\textit{et al}.\@ }
\hypersetup{
	colorlinks=true,
	linkcolor=blue,
	urlcolor=red,
	citecolor=green,
	linkbordercolor={0 0 1}
}
\newcommand{\norm}[1]{\left\lVert#1\right\rVert}
\newcommand{\argmin}{\mathop{\mathrm{argmin}}}

\usepackage{algorithm}
\usepackage{algorithmic}

\def\R{\mathbb{R}}
\def\N{\mathbb{N}}

\usepackage{comment}

\allowdisplaybreaks[1]

\newcommand{\st}{\mathop{\mathrm{subject\,\,to}}}

\renewcommand{\vec}[1]{\bm{#1}}
\newcommand{\ind}[1]{\boldsymbol{1}\left(#1\right)}
\newcommand{\inv}[1]{\frac{1}{#1}}

\def\etal{\textit{et al.}}
\def\R{\mathbb{R}}

\newtheorem{definition}{Definition}
\newtheorem{assumption}{Assumption}

\title{Distributed Online Optimization with Byzantine Adversarial Agents\thanks{This work has been partially supported by DST-INSPIRE Faculty Grant, Department of Science and Technology (DST), Govt. of India (ELE/16-17/333/DSTX/RACH).}}
\date{}
\author[1]{Sourav Sahoo}
\author[1]{Anand Gokhale}
\author[1]{Rachel Kalpana Kalaimani}
\affil[1]{Department of Electrical Engineering, IIT Madras}
\affil[ ]{\texttt {rachel@ee.iitm.ac.in}}

\begin{document}
\maketitle


\begin{abstract}
  We study the problem of non-constrained, discrete-time, online distributed optimization in a multi-agent system where some of the agents do not follow the prescribed update rule either due to failures or malicious intentions. None of the agents have prior information about the identities of the faulty agents and any agent can communicate only with its immediate neighbours. At each time step, a locally Lipschitz strongly convex cost function is revealed locally to all the agents and the non-faulty agents update their states using their local information and the information obtained from their neighbours. We measure the performance of the online algorithm by comparing it to its offline version, when the cost functions are known \textit{apriori}. The difference between the same is termed as regret. Under sufficient conditions on the graph topology, the number and location of the adversaries, the defined regret grows sublinearly. We further conduct numerical experiments to validate our theoretical results. 
\end{abstract}

\section{INTRODUCTION}
In recent years, the emphasis on identifying decentralized optimization algorithms for distributed systems has gained much traction. Many problems in network systems may be posed in the framework of distributed optimization. Some common applications appear in problems involving sensor networks~\cite{SensorsRabbat2004}, localization and robust estimation~\cite{Durham2012}, and power networks~\cite{GarciaPower2012}.

In the classical distributed optimization problem, a network of agents attempts to minimize a cost function collaboratively. This cost function is given by a sum of cost functions which are only locally accessible to each agent. There is a vast amount of literature detailing approaches to solve the distributed optimization problem. These approaches are summarized in \cite{Yang2019survey} and the references therein. The classical distributed optimization problem assumes that the local cost function is fixed throughout the duration of the problem. However, in dynamically changing environments, the objective function of each agent may be time-varying. For example, in a tracking problem, the sensor readings may be influenced by noise. This problem can be tackled under the domain of online optimization. In the online optimization problem, at each time-step, an agent ``plays'' a vector $x(t)$. The environment then ``reveals'' the cost function $f_t(\cdot)$ and the agent incurs a cost $f_t(x(t))$. In such problems, the objective is to minimize the difference between the accumulative cost incurred and the cost incurred by a hypothetical agent which had knowledge about all the objective functions \textit{apriori}. Notably, methods such as dual averaging~\cite{hosseini2013online}, mirror descent~\cite{shahrampour2017distributed}, push sum~\cite{akbari2015distributed} have been used to solve online distributed optimization problems.

Given the large scale and safety-critical applications of distributed optimization based algorithms in many different engineering problems, there is a need to develop algorithms robust to adversarial attacks, where some agents in the system may be compromised. Recent studies have considered the effect of adversaries on consensus-based distributed optimization problems~\cite{su2015byzantine,sundaram2018distributed,kuwaranancharoen2020byzantine}. Although, it is not possible to identify the exact optimal point under such circumstances, the filtering algorithm presented in does give certain performance guarantees in the adversarial case~\cite{sundaram2018distributed}.

In our work, we consider the problem of online optimization, in the presence of byzantine adversaries. Our contributions are summarized as follows
\begin{itemize}
    \item We formulate the problem of online distributed optimization in the presence of byzantine adversaries. To our knowledge, we are the first paper to consider this problem
    \item We motivate and define a notion for regret based on the behavior of the adversarial agents. We show that the regret is sublinear for any finite time horizon $T$ and grows as $\mathcal O((\ln T)^2)$.

\end{itemize}
The paper is organized as follows: we discuss the relevant preliminaries and notations in \cref{sec:notations}. We formally describe the problem statement in \cref{sec:problem_statement} and the main results are detailed in \cref{sec:main}. The experimental
results are in \cref{sec:numerical_experiments}, and the conclusions follow in
\cref{sec:conclusion}. 

\textit{Notation:} We denote the set of real numbers, non-negative reals, and natural numbers by $\R$, $\R_{\geq0}$ and $\N$ respectively. The set of all $m \times n$ real-valued matrices is denoted by $\R^{m\times n}$ and $\norm{\cdot}$ is the Euclidean norm unless stated otherwise. $[k]$ denotes the set $\{1,2,\dots, k\}$ for $k\in\N$. A stochastic vector is a vector with non-negative numbers that add up to one. We denote $B(x, r)=\{y\in\R^n|\norm{x-y}\leq r\}$, the closed ball of radius $r$ centred at $x$. For any statement $X$, $\ind{X}$ is the indicator function which is 1 if $X$ is true and 0 otherwise. 

\section{PRELIMINARIES}\label{sec:notations}
\subsection{Graph Theory}
The communication network across the agents in a distributed setting is depicted using a graph $\mathcal G$.  An undirected graph $\mathcal G = (\mathcal V, \mathcal E)$ consists of a vertex set $\mathcal V$ and an edge set $\mathcal E \subseteq \mathcal V \times \mathcal V$. A graph is said to be undirected when every edge is bidirectional i.e. if $(i,j) \in \mathcal E$, then $(j,i) \in \mathcal E$. Each vertex represents an agent, so in a system with $n$ agents, $|\mathcal V| = n$. On indexing the agents from $\{1,2,\hdots,n\}$, a graph may be characterised using an adjacency matrix $A \in \R^{n \times n}_{\geq 0}$. This matrix is constructed such that $A_{ij} > 0$ iff $\{i,j\}\in \mathcal E$. For an undirected graph $A = A^\top$. A path from agent $i$ to agent $j$ is a  sequence of agents $v_{k_1},v_{k_2} ,\hdots,v_{k_l}$ such that $v_{k_1} = i, v_{k_l} = j$ and $(v_{k_r},v_{k_{r+1}}) \in \mathcal E$ for $1\leq r \leq l-1$. A graph is said to be connected if there exists a path between any two distinct vertices. The set of neighbours of an agent $i$ are defined as $\mathcal N_i = \{j \in \mathcal V \mid (i,j) \in \mathcal E\}$.

Next, we define some properties associated with graphs as presented in \cite{sundaram2018distributed}, which will be used to define constraints on the network structure in our problem formulation.

\begin{definition} [r-Reachable set]
For a given $r \in \N$, a subset of vertices $\mathcal S \in \mathcal V$ is said to be $r$-reachable if there exists a vertex $i \in \mathcal S$ such that $\mid \mathcal N_i \backslash S\mid \geq r$.
\end{definition}

\begin{definition} [r-robust graphs]
For some $r \in \N$, a graph $\mathcal G$ is said to be $r$-robust if for all pairs of disjoint nonempty subsets $\mathcal S_1,\mathcal S_2 \subset \mathcal V$, at least one of $\mathcal S_1$, $\mathcal{S}_2$ is $r$-reachable.
\end{definition}

\subsection{The Adversarial Model}

We assume that the set of adversarial agents is fixed, and the agents do not follow any prescribed algorithm. Further, these agents are capable of sending different values to each of their neighbours. This behaviour is referred to as Byzantine adversarial behavior in literature~\cite{Fischer1985Byzantine}. We also do not assume that the adversarial agents follow a certain pattern to achieve a goal, ensuring that our adversarial model is as general as possible and robust to all potential attacks. Regarding the distribution and the topology of the adversarial agents, we assume that each agent has at most $F$ adversaries among its neighbours. This is termed as $F$-local distribution of adversaries~\cite{leblanc2013resilient}. Formally, for each agent $i$, we assume that $|\mathcal N_i \cap \mathcal A | \leq F$.

\section{PROBLEM STATEMENT}\label{sec:problem_statement}

Consider a set of $N \geq 2$ agents, interacting via a network modelled by a graph $\mathcal{G}(\mathcal{V}, \mathcal{E})$. Let $\mathcal{V}=[N]$. Let $\mathcal{A}\subset\mathcal{V}$ be the set of Byzantine adversaries. Let the non-adversarial agents be denoted by $\mathcal{R} = \mathcal{V}\backslash\mathcal{A}$. Suppose there are $R\leq N$ non-adversarial agents. Without loss of generality, assume $\mathcal{R}=[R]$. None of the non-adversarial agents have knowledge regarding the identities of the adversarial agents. At each time step $t$, a regular agent $i$, chooses its state $x_i(t)\in \R$ based on some proposed algorithm. Each agent $i\in\mathcal{V}$ has access to a sequence of locally strongly convex cost functions $f_t^i: \R \rightarrow \R$, where $f_t^i$ is revealed to agent $i$ only at the \textit{end} of each time step $t\in[T]$, where $T$ is the time horizon. 

We first discuss the offline version of multi-agent optimization in the presence of adversaries. It has been shown that there exists no algorithm such that 
\begin{align}\label{eq:first}
    \min_{x\in\R}\inv{R}\sum_{i=1}^R f_i(x)
\end{align}{}
is solvable, where $f_i(\cdot)$'s are the local cost functions corresponding to the regular agents~\cite{su2015byzantine}. We say a problem is \textit{solvable} if there exists an algorithm that obtains the optimal point which satisfies all the constraints. Hence, a relaxed version of the problem has been proposed in~\cite{su2015byzantine} where a convex combination of the objective functions under additional constraints is minimized as opposed~\eqref{eq:first}. Ideally, we would like the convex combination coefficients, $\alpha_i=\inv{R},i\in\mathcal{R}$. However, it is possible that several elements of $\alpha$ are not non-zero. So, additional parameters $\beta$ and $\gamma$ were introduced to control the ``quality'' of $\alpha$, i.e., for some $\gamma\in\mathcal{R}$, at least $\gamma$ elements of $\alpha$ are lower-bounded by $\beta>0$.  This is formally stated in \eqref{eq:relaxed_opt}.
\begin{align}\label{eq:relaxed_opt}
    &\Tilde{x}\in\argmin_{x\in\R}\sum_{i=1}^R \alpha_if_i(x)\\
    \st\quad & \alpha_i\geq0 \text{ and }\sum_{i=1}^R \alpha_i=1, \forall i\in\mathcal{R},\nonumber
    \\\quad&\sum_{i=1}^R\ind{\alpha_i\geq\beta}\geq \gamma\nonumber
\end{align}

In the distributed online convex optimization setting, in the absence of adversaries, the performance of an algorithm is measured in terms of regret defined as follows. An 
\textit{agent's regret}~\cite{akbari2015distributed, hosseini2013online} is measured as the difference between the actual cost incurred and the optimal choice in hindsight, i.e, for agent $j$, 
\begin{align}\label{eq:offline_regret}
    \textsf{Reg}_T^j = \sum_{t=1}^T\sum_{i=1}^N f_t^i(x_j(t))-\sum_{t=1}^T\sum_{i=1}^N f_t^i(x^*)
\end{align}
where
\begin{align}\label{eq:offline_no_adv}
    x^* \in \argmin_{x\in\R} \sum_{t=1}^T\sum_{i=1}^N f_t^i(x)
\end{align}
The problem in \eqref{eq:offline_no_adv} is solvable~\cite{duchi2011dual, nedic2009distributed, nedic2014distributed}. Next, we provide a notion of regret when there are adversarial agents in the network. 

Combining the idea of solvability of the offline version of the problem and the conventional definition of regret in \eqref{eq:offline_regret}, we define \textit{agent regret} and \textit{network regret} similar to~\cite{akbari2015distributed}. Define $Y_T^{\beta, \gamma}$ as:
\begin{align}
    Y_T^{\beta, \gamma}:=&\left\{x : x\in\argmin_{x\in\R}\sum_{t=1}^T\sum_{j=1}^R \alpha_j(t)f_t^j(x), \right.\\\nonumber
    &\quad \left. 0\leq \alpha_i(t)\leq 1,  \sum_{i=1}^R\alpha_i(t)=1 ~\forall i\in\mathcal{R}, \forall t\in[T]\right.\\\nonumber
    &\quad\left.\sum_{i=1}^R\ind{\alpha_i(t)\geq\beta}\geq \gamma, \forall t\in[T]\right\}
\end{align}

\begin{definition}[Agent's Regret] 
Consider a sequence of cost functions $\{f_t^1,f_t^2,\dots, f_t^R \}_{i=1}^R$ and stochastic vectors $\{\alpha(t)\}_{t=1}^T$. Then, $\forall j\in \mathcal{R}$, for any $z^*\in Y_T^{\beta, \gamma}$, the agent's regret bound is given as:
\begin{align}\label{eq:agentregret}
    \textsf{Reg}_{\alpha,T}^j = \sum_{t=1}^T\sum_{i=1}^R \alpha_i(t)f_t^i(x_j(t))-\sum_{t=1}^T\sum_{i=1}^R \alpha_i(t)f_t^i(z^*),
\end{align}
\end{definition} 

\begin{definition}[Network Regret] 
Consider a sequence of cost functions $\{f_t^1,f_t^2,\dots, f_t^R \}_{i=1}^R$ and stochastic vectors $\{\alpha(t)\}_{t=1}^T$. Then, $\forall j\in \mathcal{R}$, for any $z^*\in Y_T^{\beta, \gamma}$, the network regret bound is given as:
\begin{align}\label{eq:networkregret}
    \textsf{Reg}_{\alpha, T} = \sum_{t=1}^T\sum_{i=1}^R \alpha_i(t)f_t^i(x_i(t))-\sum_{t=1}^T\sum_{i=1}^R \alpha_i(t)f_t^i(z^*),
\end{align}
\end{definition}

Note that in the regret definition for offline case in \eqref{eq:relaxed_opt}, the co-efficient vectors $\alpha$ are fixed. But for the online case since the objective functions change at each time step, the co-efficient vectors are assumed to be time-varying as given in
\eqref{eq:agentregret} and \eqref{eq:networkregret}.


Before we proceed with our main results, we make the following assumptions regarding the nature of the objective functions and the communication model.

\begin{assumption}\label{Assumption:Function} We consider the following assumptions regarding the objective functions:
 \begin{enumerate}
 \item All the (sub)-gradients $g$ are bounded, i.e., $\norm{g}\leq L, \forall t\in[T], i\in[N]$. This implies that $f_{t}^i$ is $L$-Lipschitz, i.e., $|f(x)-f(y)|\leq L\norm{x-y}$.
     \item $f_t^i(\cdot)$ is $\rho$-strongly convex, $\forall t\in[T], i\in[N]$ in $B(0, K_1)$ and $\cup_{i=1}^R\cup_{t=1}^T\argmin f_t^i\in B(0, K_2)$ where $K_1$ and $K_2$ are constants defined similarly as in~\cite{akbari2015distributed}.
 \end{enumerate}
 \end{assumption}
 \begin{assumption} We consider the following assumptions regarding the communication model: \label{Assumption:Network}
 \begin{enumerate}
 
     \item The underlying graph representing the network is static and undirected. 
     
     \item The set of adversarial agents $\mathcal{A}$ remains fixed for all the time steps.

     \item $F$-local Byzantine model of adversarial attack.
     \item The network is $(2F+1)$-robust.
     \item Each non zero value in the adjacency matrix describing the graph is lower bounded by some  $\kappa>0$.
 \end{enumerate}
 \end{assumption}

\section{MAIN RESULTS}
\label{sec:main}
We first present the optimization algorithm in \cref{alg:opt}. Most of the existing literature in distributed optimization involving adversaries have a \textit{filtering} step included in the algorithm~\cite{su2015fault, sundaram2018distributed, kuwaranancharoen2020byzantine} where the non-faulty agents sort the received values and reject the top $k$ and bottom $k$ values for some $k\in\N$. If there are less than $k$ values higher~(or respectively lower) than agent's value, it removes all such values. The intuitive idea is to reject the outlier values, which are more likely to disrupt the consensus step. We use distributed gradient descent for its simplicity and ease of implementation.


\begin{algorithm}
\caption{Byzantine-Resilient Online Distributed Gradient Descent}
\label{alg:opt}
\begin{algorithmic}
\STATE For each $i \in \mathcal R$, initialize $x_i(0)$.\\
\FOR{$t=1 \text{ to } T$ }
\STATE
Obtain ${\{f_t^i(x_i(t)), g_i(t)\}, g_i(t)\in\partial f_t^i(x_i(t))}$ from environment.

\STATE Send $x_i(t)$ to all neighbours.
\STATE  Sort the values obtained from neighbouring agents $\mathcal{N}_i$.
\STATE $\mathcal{U}_i(t) \gets $ Set of agents that sent the top $F$ values.
\STATE $\mathcal{L}_i(t) \gets $ Set of agents that sent the bottom $F$ values.
\STATE $\mathcal{J}_i(t)\gets (\mathcal{N}_i \backslash (\mathcal{L}_i(t) \cup \mathcal{U}_i(t))) \cup \{i\}$.

\STATE Update local state as 
\begin{equation}
    x_i(t+1) = \inv{|\mathcal{N}_i|-2F+1} \left(\sum_{j \in \mathcal{J}_i(t)}x_j(t)\right) -\eta(t) g_i(t) \label{eqn:UpdateLaw}
\end{equation}{}
\ENDFOR
\end{algorithmic}
\end{algorithm}



To do a mathematical analysis of \cref{alg:opt}, we need to represent the update law in \eqref{eqn:UpdateLaw} in an expression that involves only the non-faulty agents.
\begin{proposition}[{\cite[Proposition 5.1]{sundaram2018distributed}},\cite{vaidya2012matrix}]
Consider the network $\mathcal{G}=(\mathcal V,\mathcal E)$, with a set of regular nodes $\mathcal R$ and a set of adversarial nodes $\mathcal A$. Suppose that $\mathcal A$ is an $F$-local set, and that each regular node has at least $2F+1$ neighbors. Let $\vec{x}(0)\in\R^R$ denote the initial states of all non-faulty agents and $\vec{x}(t)$ denote their states at time step $t$. Then, the update rule \eqref{eqn:UpdateLaw} for each node $i\in\mathcal R$ is mathematically equivalent to 
\begin{align}\label{eq:UpdateLawEquiv}
    x_i(t+1) = M_i(t)\vec{x}(t) -  \eta(t)g_i(t)
\end{align}
where $M_i(t)$ is a row vector that satisfies
\begin{enumerate}
    \item $M_i(t)$ is a stochastic vector, i.e,  $\sum_{j=1}^R M_{ij}(t)=1$.
    \item $M_{ij}(t)\neq0$ \emph{only if} $(i,j)\in\mathcal{E}$ or $i=j$.
    \item $M_{ii} \geq \kappa$ and at least $|\mathcal N_i| - 2F$ of the other weights are lower bounded by $\frac{\kappa}{2}$ for some $\kappa>0$.
\end{enumerate}
It should be noted that $M_i(t)$ can depend on $\vec{x}(t)$ and the behaviour of the adversarial agents.
\end{proposition}
So, the update law in \eqref{eq:UpdateLawEquiv} can be written for all the agents in a matrix form as
\begin{align}
    \vec{x}(t+1) = M(t)\vec{x}(t)-\eta(t)\vec{g}(t)
\end{align}
where $\vec{g}(t)=[g_1(t), g_2(t), \dots, g_R(t)]^\top$ and ${M(t)=[M_1(t)^\top, M_2(t)^\top, \dots, M_R(t)^\top]^\top}$.

Let $\vec{\Phi}(t,s) = \prod_{i=s}^{t}M(i)$ with $\vec{\Phi}(t,t)=M(t)$. Then, from \cite{nedic2010constrained}, 
\begin{equation}\label{prop:Phi_prop}
    \lim_{t\geq s, t\to\infty} \vec{\Phi}(t,s) = \mathbf{1}\vec q(s)^\top
\end{equation}{}
where $\vec{q}(s)$ is a stochastic vector. We now present three lemmas which are crucial for the main result of the paper. 
\begin{lemma}\label{lem:red_graph}
A reduced graph $\mathcal{H}$ of a graph $\mathcal{G}(\mathcal{V}, \mathcal{E})$ is defined as a subgraph obtained by removing all the faulty nodes from $\mathcal{V}$ and additionally removing up to $F$ edges at each non-faulty agent. For a graph satisfying \cref{Assumption:Network}, each of its reduced graphs is connected and has at least $\gamma\geq F+1$ nodes.
\end{lemma}
\begin{proof}
Let $\mathcal{G}_\mathcal{R}(\mathcal{R}, \mathcal{E}_\mathcal{R})$ denote the subgraph of $\mathcal{G}$ consisting only non-faulty nodes. Then, for a network $\mathcal{G}$ satisfying \cref{Assumption:Network}, $\mathcal{G}_\mathcal{R}$ is $(F+1)$-robust~\cite{leblanc2013resilient}. So, trivially, the number of nodes in $\mathcal{G}_\mathcal{R}$ is at least $F+1$. Furthermore, if a graph is $r$-robust, then the resulting graph after removing upto $r-1$ edges from each node is connected~\cite{sundaram2018distributed}. Combining both the statements, we conclude that the reduced graph of $\mathcal{G}$ is connected with at least $F+1$ nodes. 

\end{proof}

From \cite[Lemma 5]{su2015fault}, we have that for any fixed $s$, there exists at least $\gamma$~(as defined in \cref{lem:red_graph}) elements in $\vec{q}(s)$ that are lower bounded by $\xi^R$, for some $\xi\in(0,1)$, i.e,
\begin{align}\label{eq:q_LB}
    \sum_{i=1}^R\ind{q_i(s)\geq\xi^R}\geq\gamma
\end{align}
Define 
\begin{equation}
\label{eqn:consensus_dyn}
    y(t) := \langle \vec{q}(t), \vec{x}(t)\rangle
\end{equation}
as the \textit{convex combination} of the current states. It mimics the ``average'' state $\Bar{x}(t)=\inv{N}\sum_{i=1}^N x_i(t)$ considered in the case of distributed optimization without adversarial nodes. Furthermore, it is not difficult to show that the update rule for $y(t)$ is given by
\begin{align}\label{eq:update_y}
    y(t+1) &= y(t) - \eta(t)\vec{q}(t+1)^\top \vec{g}(t)
\end{align}
\begin{lemma}\label{lem:xi-yt_ub}
Consider the network $\mathcal G=(\mathcal{V},\mathcal E)$. Suppose that the convex functions $f_t^i,\ i\in \mathcal V$ are $L$-Lipschitz. Let the update law be given by:
\begin{align}
    \vec{x}(t+1) = M(t)\vec{x}(t) - \eta(t) \vec{g}(t) \label{eqn:UpdateLawProjection}
\end{align}

Suppose, there exists a constant $\kappa > 0$ such that at each timestep $t\in [T]$, the diagonal elements of the weight matrix $M(t)$ is lower bounded by $\kappa$ and the network contains a rooted subgraph whose edge weights are lower bounded by $\kappa$. Let $y(t)$ be the sequence defined as per \eqref{eqn:consensus_dyn}. If $\eta(t) \rightarrow 0$ as $t\rightarrow \infty$, then
    \begin{equation}
    \begin{split}
        \norm{x_i(k) - y(k)} &\leq C\theta^{k-1}\sum_{j=1}^R\norm{x_j(0)} \\
        &+RCL\sum_{r=0}^{k-2}\eta(r)\theta^{k-r-2} +2\eta(k-1)L
    \end{split}
    \end{equation}
    for some $C>0$ and $\theta\in[0,1)$. It is to be noted that the upper bound is independent of $i$ and depends only on $k$. We denote this upper bound by $\zeta(k)$ which we refer several times later in this paper.
\end{lemma}
\begin{proof}
Consider the dynamics of $y(k)$, and use the update law from the equation \eqref{eqn:UpdateLawProjection}.
\begin{align*}
    y(k+1)  &=  \vec q(k+1)^\top \vec x(k+1)\\
    &=  \vec q(k+1)^\top (M(k)\vec x(k) - \eta(k) \vec g(k)))\\ &\stackrel{(a)}{=}  \vec q(k)^\top \vec x(k) - \eta(k)\vec q(k+1)^\top \vec g(k)\\
    \implies y(k+1) &=y(k)- \eta(k) \vec q(k+1)^\top \vec g(k)\stepcounter{equation}\tag{\theequation}\label{eq:y-dyanmics}
\end{align*}{}
where $(a)$ holds because by definition of $\vec{\Phi}$, we have $\vec q(s)^\top = \vec q(s+1)^\top M(s)$.
 Considering the dynamics of $x_i(k)$ and $y(k)$ over several time steps, starting from  time $s$, and ending at time $k+1$, 
 For $i\in\mathcal{R}$,
 \begin{align}
  x_i(k+1) &= [\vec{\Phi}(k, s)\vec{x}(s)]_i\nonumber\\
  &\quad-\sum_{r=s}^{k-1}\eta(r)\sum_{j=1}^R\vec{\Phi}(k, r+1)_{ij}g_j(r)-\eta(k)g_i(k)
 \end{align}
 Similarly, using the update step in \eqref{eq:y-dyanmics} for $y(k')$ for $s\leq k'\leq k+1$ recursively,
\begin{align*}
    y(k+1)&= \vec{q}(s)^\top\vec{x}(s) - \sum_{r=s}^{k-1}\eta(r) \sum_{j=1}^R q_j(r+1)g_j(r)\nonumber\\& -\eta(k)\sum_{j=1}^R q_j(k+1) g_j(k)
\end{align*}
Setting $s = 0$, and using the triangle law,
\begin{align*}
    &\norm{x_i(k) - y(k)} \\
    &\leq \norm{\sum_{j=1}^R x_j(0)(\vec{\Phi}(k-1,0)_{ij} - q_j(0))} \\
    &\quad+ \sum_{r=0}^{k-2}\eta(r)\norm{\sum_{j=1}^R g_j(r)(\vec{\Phi}(k-1,r+1)_{ij} - q_j(r+1))} \nonumber \\ & +\eta(k-1)\norm{g_i(k-1)} + \eta(k-1) \norm{\sum_{j=1}^R q_j(k) g_j(k-1)} \nonumber\\
    &\leq \sum_{j=1}^R\norm{x_j(0)}\norm{\vec{\Phi}(k-1,0)_{ij} - q_j(0)}\\
    &\quad+\sum_{r=0}^{k-2}\eta(r)\sum_{j=1}^R\norm{\vec g_j(r)}\norm{\vec{\Phi}(k-1,r+1)_{ij} - q_j(r+1)} \nonumber \\ 
    &\quad+\eta(k-1)\norm{g_i(k-1)} + \eta(k-1) \sum_{j=1}^R q_j(k)\norm{g_j(k-1)}\stepcounter{equation}\tag{\theequation}\label{eq:xiytub}
\end{align*}
From Nedic~\etal~\cite{nedic2010constrained}, we get that for some $C>0,~\theta\in[0,1)$, $\norm{\vec{\Phi}(k,s)_{ij}-q_j(s)} \leq C\theta^{k-s}$. Further, $\norm{g_i(k)}\leq L$ and $\sum_{j=1}^R q_j(k)=1$. Hence, by upper bounding the terms of \eqref{eq:xiytub} appropriately, for $k\geq2$, we get the result of \cref{lem:xi-yt_ub}.

\end{proof}

\begin{lemma}\label{lem:sublinearity}
Consider the conditions mentioned in \cref{lem:xi-yt_ub}. Then, for learning rate $\eta(t)=\inv{\rho t}, t\geq1$, $\eta(0)=0$ and any finite time horizon $T$,
\begin{equation}\label{eq:lem-sublinearity}
\begin{split}
&\sum_{t=1}^T\norm{x_i(t) - y(t)} \leq C_1 + C_2(1+\ln T), \forall i, \text{ where }\\
&C_1 = \frac{C}{1-\theta}\sum_{j=1}^R\norm{x_j(0)},~~
C_2 = \frac{2L}{\rho} + \frac{RCL}{\rho(1-\theta)}
\end{split}
\end{equation}
\end{lemma}

\begin{proof}
Consider $\eta(t)=\inv{\rho t},t\geq1$ and $\eta(0)=0$. Observe,
\begin{align}\label{eq:sum_of_steps}
    \sum_{s=1}^t \eta(t) \leq \inv{\rho}\left(1+\int_{1}^t\inv{z}dz\right)= \inv{\rho}(1+\ln t)
\end{align}


If $\eta(t)=\inv{\rho t}$ and $\theta\in[0, 1)$, then,
\begin{align}
    &\sum_{k=1}^t \sum_{r=0}^{k-2}\eta(r)\theta^{k - r - 2}\leq\sum_{r=0}^{t}\eta(r)\sum_{s=0}^\infty \theta^{s}\stackrel{\eqref{eq:sum_of_steps}}{\leq}\frac{(1+\ln t)}{\rho (1-\theta)}\label{eq:sum_of_steps_2}
\end{align}

\begin{align*}
    &\sum_{k=1}^t\norm{x_i(k) - y(k)}\\ &\leq \sum_{k=1}^t\left\{C\theta^{k-1}\sum_{j=1}^R\norm{x_j(0)}\right.\\
    &\quad\left.+ RCL\sum_{r=0}^{k-2}\eta(r)\theta^{k-r-2} +2\eta(k-1)L\right\}\\
    &\leq \frac{C}{1-\theta}\sum_{j=1}^R\norm{x_j(0)}\\
    &\quad+\sum_{k=1}^t\left\{ RCL\sum_{r=0}^{k-2}\eta(r)\theta^{k - r - 2} +2\eta(k-1)L\right\}\\
    &\stackrel{\eqref{eq:sum_of_steps}}{\leq} \frac{C}{1-\theta}\sum_{j=1}^R\norm{x_j(0)}+\frac{2L}{\rho}(1+\ln t)\\
    &\quad+RCL\sum_{k=1}^t \sum_{r=0}^{k-2}\eta(r)\theta^{k - r - 2}\\
    &\stackrel{\eqref{eq:sum_of_steps_2}}{\leq} \frac{C}{1-\theta}\sum_{j=1}^R\norm{x_j(0)}+\frac{2L}{\rho}(1+\ln t)+\frac{RCL(1+\ln t)}{\rho(1-\theta)}
\end{align*}
Hence, by grouping the constant terms and the coefficients of $(1+\ln t)$, we get the result.

\end{proof}

We state a theorem regarding the sublinearity of the network regret in \cref{thm:bigthm}. 

\begin{theorem}[Sublinear Network Regret Bound]\label{thm:bigthm} Under \cref{Assumption:Function} and \cref{Assumption:Network}, with a learning rate $\eta(t)=\inv{\rho t}$, the network regret defined in \eqref{eq:networkregret} with $\alpha(t)=\vec{q}(t+1)$, defined in \eqref{prop:Phi_prop}, is sublinear. Precisely,
  \begin{align*}
    \textsf{Reg}_{\alpha,T}&\leq A_1 + A_2(1+\ln T) + A_3(1+\ln T)^2
\end{align*}
where
\begin{equation}\label{eq:a1a2a3}
    \begin{split}
        A_1&=LC_1+\rho C_1\norm{y(0) -z^*}+\frac{\rho}{2}\norm{y(0) -z^*}^2\\
        A_2&=L(C_1+C_2)+\frac{L^2}{2\rho}+(L+\rho C_2)\norm{y(0) -z^*}\\
        A_3&=\frac{L^2}{2\rho}+LC_2
    \end{split}
\end{equation}
and $C_1$ and $C_2$ are the constants mentioned in \cref{lem:sublinearity} and $z^*\in Y_T^{\beta, \gamma}$. 
\end{theorem}
\begin{proof}
Let $z^*\in Y_T^{\beta, \gamma}$ and $\alpha(t)=\vec{q}(t+1)$. By definition of $\textsf{Reg}_{\alpha, T}$,
\begin{align*}
    &\sum_{t=1}^T\left\{\sum_{j=1}^R\alpha_j(t)(f_t^j(x_j(t))-f_t^j(z^*))\right\}\\
    &\leq \sum_{t=1}^T\left\{\sum_{j=1}^R\alpha_j(t)\left(\left\langle g_j(t), x_j(t)-z^*\right\rangle-\frac{\rho}{2}\norm{x_j(t)-z^*}^2\right)\right\}\\
    &= \sum_{t=1}^T\left\{\sum_{j=1}^R\alpha_j(t)\left\langle g_j(t), x_j(t)-y(t)\right\rangle\right.\\
    &\quad+\left.\sum_{j=1}^R\alpha_j(t)\left\langle g_j(t), y(t)-z^*\right\rangle\right.\\
    &\quad-\left.\frac{\rho}{2}\sum_{j=1}^R\alpha_j(t)\norm{x_j(t)-y(t)+y(t)-z^*}^2\right\}\\
    &\leq \sum_{t=1}^T\left\{\sum_{j=1}^R\alpha_j(t) L\norm{x_j(t)-y(t)}\right.\nonumber\\
    &\quad-\left.\frac{\rho}{2}\sum_{j=1}^R\alpha_j(t)(\norm{x_j(t)-y(t)}^2+\norm{y(t)-z^*}^2\right.\\
    &\quad+\left. 2\langle x_j(t)-y(t), y(t)-z^*\rangle)\right.\\
    &\quad\left.+\sum_{j=1}^R\alpha_j(t)\left\langle g_j(t), y(t)-z^*\right\rangle\right\}\\
    &\leq L\sum_{t=1}^T\left\{\sum_{j=1}^R\alpha_j(t)\norm{x_j(t)-y(t)}\right\}\nonumber\\
    &\quad+\sum_{t=1}^T\left\{\left\langle \sum_{j=1}^R\alpha_j(t)g_j(t), y(t)-z^*\right\rangle\right\}\nonumber\\
    &\quad-\frac{\rho}{2}\sum_{t=1}^T\left\{\norm{y(t)-z^*}^2\right.\\
    &\quad\left.+2\sum_{j=1}^R\alpha_j(t)\langle x_j(t)-y(t), y(t)-z^*\rangle\right\}\stepcounter{equation}\tag{\theequation}\label{eq:strongcvx}
\end{align*}
\subsection{Bounding the first term of \eqref{eq:strongcvx}}
The first term can be bounded directly as follows:
\begin{align*}
    &L\sum_{t=1}^T\left(\sum_{j=1}^R \alpha_j(t)\norm{x_j(t)-y(t)}\right)\leq L\sum_{t=1}^T\sum_{j=1}^R \alpha_j(t)\zeta(t)\\
    &=L\sum_{t=1}^T\zeta(t)\stackrel{\eqref{eq:lem-sublinearity}}{\leq} LC_1 + LC_2(1+\ln T)\stepcounter{equation}\tag{\theequation}\label{eq:1_bounded}
\end{align*}
\subsection{Bounding the second and third term of \eqref{eq:strongcvx}}
For non-constrained optimization,
\begin{align*}
    &\norm{y(t+1) - z^*}^2-\norm{y(t) - z^*}^2\\
    &= \norm{y(t) - \eta(t)\langle q(t+1), g(t)\rangle - z^*}^2-\norm{y(t) - z^*}^2\\
    &= \norm{\eta(t)\langle q(t+1), g(t)\rangle}^2\\
    &\quad- 2\eta(t)\langle\langle q(t+1), g(t)\rangle,y(t)-z^*\rangle\\
    &\leq \eta(t)^2L^2\norm{q(t+1)}_1^2 \nonumber\\ &\quad- 2\eta(t)\langle\langle q(t+1), g(t)\rangle,y(t)-z^*\rangle\\
    &= \eta(t)^2L^2 - 2\eta(t)\langle\langle q(t+1), g(t)\rangle,y(t)-z^*\rangle
\end{align*}
So,
\begin{align}\label{eq:27}
  &\langle\langle q(t+1), g(t)\rangle,y(t)-z^*\rangle \nonumber\\&\leq \frac{\eta(t)}{2} L^2  + \frac{\norm{y(t) - z^*}^2}{2\eta(t)}  - \frac{\norm{y(t+1) - z^*}^2}{2\eta(t)}  
\end{align}
For $\alpha(t) = \vec q(t+1)$, as mentioned in \cref{thm:bigthm}, the second term of \eqref{eq:strongcvx} is 
\begin{align*}
  &\sum_{t=1}^T\left(\left\langle \sum_{j=1}^R\alpha_j(t)g_j(t), y(t)-z^*\right\rangle\right)\\
  &=  \sum_{t=1}^T\left(\left\langle \langle q(t+1), g(t)\rangle, y(t)-z^*\right\rangle\right)\\
  &\stackrel{\eqref{eq:27}}{\leq} \sum_{t=1}^T\left(\frac{\eta(t)}{2} L^2  + \frac{\norm{y(t) - z^*}^2}{2\eta(t)}  - \frac{\norm{y(t+1) - z^*}^2}{2\eta(t)}\right)\\
  &\leq \sum_{t=1}^T\left(\frac{\eta(t)}{2} L^2  + \frac{\norm{y(t) - z^*}^2}{2\eta(t)}- \frac{\norm{y(t+1) - z^*}^2}{2\eta(t+1)}\right.\nonumber\\ &\quad+ \left.\frac{\norm{y(t+1) - z^*}^2}{2\eta(t+1)}  - \frac{\norm{y(t+1) - z^*}^2}{2\eta(t)}\right)\\
  &= \frac{L^2}{2}\sum_{t=1}^T\eta(t)+\sum_{t=1}^T\left( \frac{\norm{y(t) - z^*}^2}{2\eta(t)}- \frac{\norm{y(t+1) - z^*}^2}{2\eta(t+1)}\right)\\
  &\quad+\sum_{t=1}^T\left( \frac{\norm{y(t+1) - z^*}^2}{2\eta(t+1)}  - \frac{\norm{y(t+1) - z^*}^2}{2\eta(t)}\right)\\
  &\leq \frac{L^2}{2}\sum_{t=1}^T\eta(t)+ \frac{\norm{y(1) - z^*}^2}{2\eta(1)}\\
  &\quad+\sum_{t=1}^T\norm{y(t+1) - z^*}^2\left(\frac{1}{2\eta(t+1)}  - \frac{1}{2\eta(t)}\right)\\
  &\leq\frac{L^2}{2\rho}(1+\ln T)+ \frac{\rho}{2}\norm{y(0) - z^*}^2+\frac{\rho}{2}\sum_{t=1}^T\norm{y(t+1) - z^*}^2\stepcounter{equation}\tag{\theequation}\label{eq:second_term}
\end{align*}
The last statement holds because, by assumption, $\eta(0)=0\implies y(1)=y(0)$.
Considering the second and third term of \eqref{eq:strongcvx} jointly, 
\begin{align*}
  &\sum_{t=1}^T\left(\left\langle \sum_{j=1}^R\alpha_j(t)g_j(t), y(t)-z^*\right\rangle\right)-\frac{\rho}{2}\sum_{t=1}^T\left\{\norm{y(t)-z^*}^2\right.\\
    &\quad\left.+2\sum_{j=1}^R\alpha_j(t)\langle x_j(t)-y(t), y(t)-z^*\rangle\right\}\nonumber\\
  &\stackrel{\eqref{eq:second_term}}{\leq}\frac{L^2}{2\rho}(1+\ln T)+ \frac{\rho}{2}\norm{y(0) - z^*}^2\\
  &\quad+\frac{\rho}{2}\sum_{t=1}^T\left(\norm{y(t+1) - z^*}^2-\norm{y(t)-z^*}^2\right)\\
  &\quad-\rho\sum_{t=1}^T\left\{\sum_{j=1}^R\alpha_j(t)\langle x_j(t)-y(t), y(t)-z^*\rangle\right\}\\
  &=\frac{L^2}{2\rho}(1+\ln T)+ \frac{\rho}{2}\norm{y(T+1) - z^*}^2\\
  &\quad-\rho\sum_{t=1}^T\left\{\sum_{j=1}^R\alpha_j(t)\langle x_j(t)-y(t), y(t)-z^*\rangle\right\}\\
  &\leq\frac{L^2}{2\rho}(1+\ln T)+ \frac{\rho}{2}\norm{y(T+1) - z^*}^2\\
  &\quad+\rho\sum_{t=1}^T\left\{\sum_{j=1}^R\alpha_j(t) \norm{x_j(t)-y(t)}\norm{ y(t)-z^*}\right\}\\
  &\leq\frac{L^2}{2\rho}(1+\ln T)+ \frac{\rho}{2}\norm{y(T+1) - z^*}^2\\
  &\quad+\rho\sum_{t=1}^T\left\{\sum_{j=1}^R\alpha_j(t) \zeta(t)\norm{ y(t)-z^*}\right\}\\
  &\leq\frac{L^2}{2\rho}(1+\ln T)+ \frac{\rho}{2}\norm{y(T+1) - z^*}^2\\
  &\quad+\rho\sum_{t=1}^T\zeta(t)\norm{ y(t)-z^*}\stepcounter{equation}\tag{\theequation}\label{eq:second_and_third_term_final}
\end{align*}
Observe,
\begin{align*}
    y(s+1) - y(s)&=  - \eta(s) \langle q(s+1), g(s)\rangle\\
    \sum_{s=1}^t(y(s+1) - y(s))&=  -\sum_{s=1}^t\eta(s) \langle q(s+1), g(s)\rangle\\
    y(t+1) &= y(0)  -\sum_{s=1}^t \eta(s) \langle q(s+1), g(s)\rangle
\end{align*}
Subtracting $z^*$ on both sides and using the triangle law,
\begin{align*}
    \norm{y(t+1) -z^*} 
    &\leq \norm{y(0) -z^*}\\
    &\quad+\sum_{s=1}^t \eta(s) \sum_{j=1}^R q_j(s+1) \norm{g_j(s)}\\
    &\leq \norm{y(0) -z^*} +L\sum_{s=1}^t \eta(s)\\
    &\leq \norm{y(0) -z^*} +\frac{L}{\rho}(1+\ln t)\stepcounter{equation}\tag{\theequation}\label{eq:third_term_final}
\end{align*}
Bounding the second term of \eqref{eq:second_and_third_term_final}, 
\begin{align*}
  \norm{y(T+1) - z^*}^2&\stackrel{\eqref{eq:third_term_final}}{\leq}\left(\norm{y(0) -z^*} +\frac{L}{\rho}(1+\ln T)\right)^2\\
  &=\norm{y(0) -z^*}^2+\frac{L^2}{\rho^2}(1+\ln T)^2\\
  &\quad+\frac{2L}{\rho}\norm{y(0) -z^*}(1+\ln T)
\end{align*}
Bounding the third term of \eqref{eq:second_and_third_term_final}, 
\begin{align*}
    &\sum_{t=1}^T\zeta(t)\norm{ y(t)-z^*} \\&\stackrel{\eqref{eq:third_term_final}}{\leq}\sum_{t=1}^T\zeta(t)\left(\norm{y(0) -z^*} +\frac{L}{\rho}(1+\ln t)\right)\\
    &\stackrel{\eqref{eq:lem-sublinearity}}{\leq}(C_1+C_2(1+\ln T))\left(\norm{y(0) -z^*} +\frac{L}{\rho}(1+\ln T)\right)
\end{align*}
So, \eqref{eq:second_and_third_term_final}~(equivalently, the second and third terms of \eqref{eq:strongcvx}) is bounded by:
\begin{align}\label{eq:2_and_3_bounded}
  &\frac{L^2}{2\rho}(1+\ln T)+\frac{\rho}{2}\left(\norm{y(0) -z^*}^2+\frac{L^2}{\rho^2}(1+\ln T)^2\right.\nonumber\\
  &\quad+\left.\frac{2L}{\rho}\norm{y(0) -z^*}(1+\ln T)\right)\nonumber\\&\quad+\rho(C_1+C_2(1+\ln T))\left(\norm{y(0) -z^*} +\frac{L}{\rho}(1+\ln T)\right)
\end{align}

Combining the results of \eqref{eq:1_bounded} and \eqref{eq:2_and_3_bounded}, we complete the proof of \cref{thm:bigthm}.

\end{proof}
We now present the main result of our paper, i.e., the sublinearity of agent's regret in \cref{thm:2}. 

\begin{theorem}[Sublinear Agent's Regret Bound]\label{thm:2}
  Under \cref{Assumption:Function} and \cref{Assumption:Network}, with a learning rate $\eta(t)=\inv{\rho t}$, the regret of agent $i\in\mathcal{R}$ defined in \eqref{eq:agentregret} with $\alpha(t)=\vec{q}(t+1)$, defined in \eqref{prop:Phi_prop}, is sublinear. Precisely,
  \begin{align*}
    \textsf{Reg}^i_{\alpha, T}&\leq B_1 + B_2(1+\ln T) + B_3(1+\ln T)^2
\end{align*}
where
\begin{equation}
    \begin{split}
        B_1&=3LC_1+\rho C_1\norm{y(0) -z^*}+\frac{\rho}{2}\norm{y(0) -z^*}^2\\
        B_2&=L(C_1+3C_2)+\frac{L^2}{2\rho}+(L+\rho C_2)\norm{y(0) -z^*}\\
        B_3&=\frac{L^2}{2\rho}+LC_2
    \end{split}
\end{equation}
and $C_1$ and $C_2$ are the constants mentioned in \cref{lem:sublinearity} and $z^*\in Y_T^{\beta, \gamma}$. 
\end{theorem}
\begin{proof}
Let $z^*\in Y_T^{\beta, \gamma}$ and $\alpha(t)=\vec{q}(t+1)$. By definition of $\textsf{Reg}_{\alpha, T}^i$, 
\begin{align*}
    &\sum_{t=1}^T\left(\sum_{j=1}^R \alpha_j(t) f_t^j(x_i(t))-\sum_{j=1}^R \alpha_j(t)f_t^j(z^*)\right)\\
    &= \sum_{t=1}^T\left(\sum_{j=1}^R \alpha_j(t)(f_t^j(x_i(t))- f_t^j(y(t)))\right.\\
    &\left.+\sum_{j=1}^R \alpha_j(t)(f_t^j(y(t))-f_t^j(x_j(t))+f_t^j(x_j(t))-f_t^j(z^*))\right)\\
    &\stackrel{\eqref{eq:networkregret}}{\leq} \textsf{Reg}_{\alpha, T} \\
    &\quad+L\sum_{t=1}^T\left(\sum_{j=1}^R \alpha_j(t)\norm{x_j(t)-y(t)}+ \norm{x_i(t)-y(t)}\right)
\end{align*}
From \cref{thm:bigthm},
\begin{align}\label{eq:final_part1}
    \textsf{Reg}_{\alpha, T} \leq A_1 + A_2(1+\ln T) + A_3(1+\ln T)^2
\end{align}
for $A_1$, $A_2$ and $A_3$ described in \eqref{eq:a1a2a3}. Further,
\begin{align*}
    &L\sum_{t=1}^T\left(\sum_{j=1}^R \alpha_j(t)\norm{x_j(t)-y(t)}+ \norm{x_i(t)-y(t)}\right)\\
    &\stackrel{\eqref{eq:lem-sublinearity}}{\leq} 2LC_1 + 2LC_2(1+\ln T)\stepcounter{equation}\tag{\theequation}\label{eq:final_part2}
\end{align*}{}
where $\zeta(t)$ is defined in \cref{lem:xi-yt_ub}. Combining \eqref{eq:final_part1} and \eqref{eq:final_part2}, we get the desired result.

\end{proof}
\section{NUMERICAL EXPERIMENTS}
\label{sec:numerical_experiments}

We provide an experiment to verify our algorithm. Motivated by \cite{hosseini2013online}, we consider a network of $N$ sensors. All of these sensors observe a vector $ x \in \R^d$, which is randomly chosen. Each sensor $i \in [N]$ measures a quantity $z_i(t) \in \R^{p_i}$ at time $t$. We assume that each measurement is associated with some noise. Formally, each sensor is modelled as a linear function of $x$, i.e. $z_i(t) = H_i x + v_i$. Here, $H_i \in \R^{p_i\times d}$ is an observation matrix, with a bounded norm, and $v_i$ represents the the noise. The local estimate for $x$, given by $\hat x_i$ is used to compute a cost function. The cost function at time $t$ is given by
\begin{align*}
    f_i(t) = \frac{1}{2} \norm{z_i(t) - H_i \hat x_i}^2
\end{align*}
Clearly, the cost functions satisfy \cref{Assumption:Function}. The underlying network and locations of the adversaries are chosen such that~\cref{Assumption:Network} is satisfied. At each step, the non-faulty agents attempt to minimize regret by following~\cref{alg:opt}. The vector $\vec q(t)$ is calculated similar to the method described in~\cite{sundaram2018distributed}, and our regret is estimated as described in \eqref{eq:agentregret} and \eqref{eq:networkregret} . In an offline setting, the optimal point is given by
\begin{align*}
    x_i^* = \frac{1}{T} \sum_{t=1}^{T} \left(\sum_{i=1}^N H_i^\top H_i\right)^{-1}\left(\sum_{i=1}^{N} H_i^\top z_i(t)\right).
\end{align*}
If the noise characteristics for $v_i$ were known beforehand, it would have been possible to solve the problem in an offline mode. 
To complete our problem setup, we assume that a fixed set of agents have been compromised, either by an external attacker, or due to some adverse environmental conditions. These agents are modelled as adversaries. We do not assume any knowledge of the location of these agents.


For the purposes of our numerical simulation, we assume that $x \in \R$. For each sensor, $H_i \in \R$ is chosen from a uniform distribution ranging $(0,2)$, and $v_i$ is sampled from a normal random variable at each time instant. We consider a network of 100 agents, with 15 adversarial agents. We construct a $(2F+1)$-robust graph using the method proposed in \cite{zhang2012robustness}. The adversarial agents send conflicting and incorrect information to their neighbours, sampled from a uniform distribution. \cref{fig:Average_sublinear} shows the agent regret and network regret as defined in \eqref{eq:agentregret} and \eqref{eq:networkregret} respectively.


\begin{figure}
    \centering
    \includegraphics[width = 0.5 \textwidth]{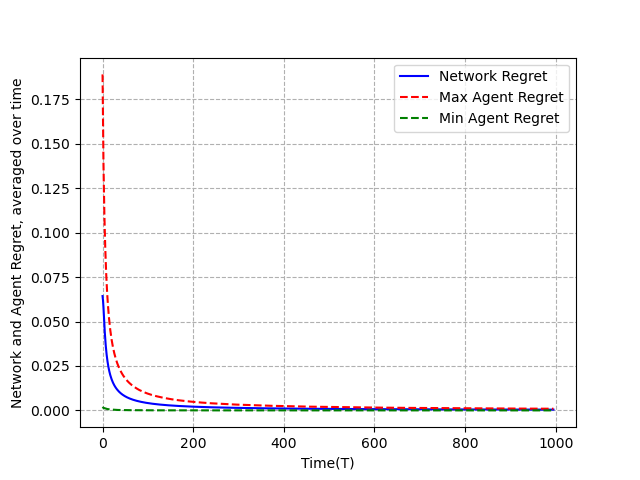}
    \caption{The Network Regret averaged over time($R(T)/T$) is presented, based on the setup described in section \ref{sec:numerical_experiments}. We also present the agent regrets for the agents with maximum and minimum regret, averaged over time. The regret is sublinear in nature, as expected based on the upper bound presented in \cref{sec:main}.}
    \label{fig:Average_sublinear}
\end{figure}

\section{CONCLUSION}\label{sec:conclusion}

In this work, we discuss the problem of distributed online optimization in the presence of Byzantine adversaries. We defined the notion of regret for this case and proved that our algorithm results in a sublinear regret bound. Currently, the coefficients that define the convex combination for the local objective functions are time dependent in nature. An interesting direction of research involves the identification of an algorithm resulting in a time invariant convex combination of cost functions. We also assume that the objective functions considered in this work are strongly convex and the regret bound obtained is $\mathcal{O}((\ln T)^2)$. A different research direction could be to consider non-strong convex functions and attain sublinear regret of the form $\mathcal{O}(T^{1-\delta}), \delta>0$.


\bibliographystyle{alpha}
\bibliography{egbib}

\end{document}